\documentclass[a4paper,10pt]{amsart}
\usepackage[T1]{fontenc}
\usepackage[utf8]{inputenc}
\usepackage{lmodern}
\usepackage[english]{babel}

\usepackage{amsmath}
\usepackage{amssymb}
\usepackage{amsbsy}
\usepackage{mathtools}
\usepackage{mathrsfs}
\usepackage{bookmark}
\usepackage{hyperref}
\hypersetup{
	colorlinks=true,
	linkcolor=blue,
	filecolor=magenta,      
	urlcolor=cyan,
}

\usepackage[alphabetic]{amsrefs}

\usepackage{enumitem}
\usepackage{graphicx}
\usepackage{tikz}
\usepackage{tikz-cd}
\usetikzlibrary{matrix,arrows,decorations.pathmorphing}
\usepackage[all,cmtip]{xy}

\newcommand{\mypar}{~\\\noindent\refstepcounter{subsubsection}\textbf{(\thesubsubsection)}\space}
\newcommand\arr{\ifinner\to\else\longrightarrow\fi}
\newcommand{\thmref}[1]{\hyperref[#1]{Theorem \ref{#1}}}
\newcommand{\propref}[1]{\hyperref[#1]{Proposition \ref{#1}}}
\newcommand{\lmref}[1]{\hyperref[#1]{Lemma \ref{#1}}}
\newcommand{\corref}[1]{\hyperref[#1]{Corollary \ref{#1}}}
\newcommand{\rmkref}[1]{\hyperref[#1]{Remark \ref{#1}}}
\newcommand{\dfref}[1]{\hyperref[#1]{Definition \ref{#1}}}
\newcommand{\parref}[1]{\hyperref[#1]{(\ref{#1})}}

\newenvironment{thm}{\vspace{8.0pt plus 2.0pt minus 4.0pt}\noindent\refstepcounter{subsubsection}\textbf{(\thesubsubsection) Theorem.}\itshape}{\vspace{8.0pt plus 2.0pt minus 4.0pt}}
\newenvironment{prop}{\vspace{8.0pt plus 2.0pt minus 4.0pt}\noindent\refstepcounter{subsubsection}\textbf{(\thesubsubsection) Proposition.}\itshape}{\vspace{8.0pt plus 2.0pt minus 4.0pt}}
\newenvironment{lm}{\vspace{8.0pt plus 2.0pt minus 4.0pt}\noindent\refstepcounter{subsubsection}\textbf{(\thesubsubsection) Lemma.}\itshape}{\vspace{8.0pt plus 2.0pt minus 4.0pt}}
\newenvironment{cor}{\vspace{8.0pt plus 2.0pt minus 4.0pt}\noindent\refstepcounter{subsubsection}\textbf{(\thesubsubsection) Corollary.}\itshape}{\vspace{8.0pt plus 2.0pt minus 4.0pt}}
\newenvironment{rmk}{\vspace{8.0pt plus 2.0pt minus 4.0pt}\noindent\refstepcounter{subsubsection}\textit{(\thesubsubsection) Remark.}}{\vspace{8.0pt plus 2.0pt minus 4.0pt}}

\newtheorem*{thm-no-num}{Theorem}
\newtheorem*{cor-no-num}{Corollary}

\newcommand{\A}{\mathbb A}

\newcommand{\Bcal}{\mathcal{B}}

\newcommand{\bD}{\textbf{D}}

\newcommand{\Ecal}{\mathcal E}
\newcommand{\Fcal}{\mathcal F}

\newcommand{\Gcal}{\mathcal G}
\newcommand{\Hcal}{\mathcal H}
\newcommand{\Ical}{\mathcal I}

\newcommand{\Lcal}{\mathcal L}
\newcommand{\Mcal}{\mathcal M}

\newcommand{\Ocal}{\mathcal O}

\newcommand{\Pcal}{\mathcal P}

\newcommand{\Qcal}{\mathcal Q}

\newcommand{\Tcal}{\mathcal T}
\newcommand{\Ucal}{\mathcal U}

\newcommand{\VV}{\mathbb V}

\newcommand{\Pic}{\text{\rm Pic}}
\newcommand{\GL}{\text{GL}}
\newcommand{\Sym}{\text{Sym}}

\newcommand{\pr}{\text{\rm pr}}

\newcommand{\im}{\text{Im}}

\newcommand{\wt}{\widetilde}
\newcommand{\ev}{\text{ev}}

\newcommand\PP{\mathbb{P}}

\renewcommand{\c}{\operatorname{c}}

\newcommand\ZZ{\mathbb{Z}}

\newcommand\into{\hookrightarrow}

\newcommand{\GG}{\mathbb{G}}
\newcommand{\Gras}{\operatorname{Gr}_m(W_{d})}
\newcommand{\dmn}{_{d,m}}

\begin{document}
\title[Picard group of moduli of curves of low genus in positive char.]{Picard group of moduli of curves of low genus in positive characteristic}
\author[A. Di Lorenzo]{Andrea Di Lorenzo}
\address{Scuola Normale Superiore, Piazza dei Cavalieri 7, 56126 Pisa, Italy}
\email{andrea.dilorenzo@sns.it}
\date{\today}

\begin{abstract}
	We compute the Picard group of the moduli stack of smooth curves of genus $g$ for $3\leq g\leq 5$, using methods of equivariant intersection theory. We base our proof on the computation of some relations in the integral Chow ring of certain moduli stacks of smooth complete intersections. As a byproduct, we compute the cycle classes of some divisors on $\Mcal_g$.
\end{abstract}
\maketitle

\section*{Introduction}
The Picard group is an important invariant of schemes. In the landmark paper \cite{Mum63} Mumford first introduced the notion of Picard group of a moduli functor, and actually computed it in the case of the moduli functor of elliptic curves. In subsequent works (e.g. \cite{Vis89}, \cite{Kre}) the notion of Picard group had been extended to a large class of algebraic stacks. Moreover, in \cite{Tot} and \cite{EG} the authors introduced the notion of equivariant Picard group of a scheme $X$ endowed with the action of an algebraic group $G$: they showed that the equivariant Picard group coincides with the Picard group of the associated quotient stack $[X/G]$.

In the papers \cite{Har} and \cite{AC} the authors computed the Picard group of the moduli stack of smooth curves of genus $g\geq 3$ over a base field of characteristic zero: in particular they proved that $\Pic(\Mcal_g)$ is a free abelian group generated by a single element $\lambda_1$, which is the first Chern class of the Hodge bundle. The proof of these results relies on topological techniques that seem hard to extend to the case of base fields of positive characteristic.

Over base fields of positive characteristic we know the \emph{rational} Chow ring of $\Mcal_g$ when $g\leq 6$ (see \cite{m3bar, m4, Iza, PV}) and the rational Picard group for every $g\geq 3$ (\cite{Mor}), which turns to be of rank $1$ and generated by $\lambda_1$.

Therefore, it may still be the possible that, in positive characteristic, the abelian group $\Pic(\Mcal_g)$ has more than one generator. In this paper we show that, in the range $3\leq g\leq 5$, this is not the case.
\begin{thm-no-num}
    The Picard group of $\Mcal_g$, for $3\leq g\leq 5$, is freely generated by $\lambda_1$, without any assumption on the characteristic of the base field.
\end{thm-no-num}
The proof of this theorem is obtained using methods of equivariant intersection theory  (\cite{Tot} and \cite{EG}). At the present moment equivariant intersection theory is the only tool available that can be used to say something on integral Chow rings of moduli stacks (see, among others, \cite{Vis98}, \cite{EF08}, \cite{EF09}, \cite{FulVis} and \cite{Dil18}). This approach to cycle theoretic questions has also the advantage of being almost independent of the characteristic of the base field, which is a key feature for the present work.

Furthermore, using these techniques, we are also able to determine along the way, almost without any extra effort, the cycle classes of some geometrically meaningful divisors on $\Mcal_g$  in terms of $\lambda_1$, recovering some of the computations contained in \cite{Tei} and \cite{HM}.
\begin{cor-no-num}
    Let $\Hcal_g$ (resp. $\Tcal_g$, $\Mcal_g^{\ev}$) denote the moduli stack of hyperelliptic curves (resp. of trigonal curves, of smooth curves with an even theta characteristic) of genus $g$. Then we have:
    \begin{enumerate}
        \item $[\Hcal_3]=9\lambda_1$.
        \item $[\Mcal_4^{\rm ev}]=34\lambda_1$.
        \item $[\Tcal_5]=8\lambda_1$.
    \end{enumerate}
\end{cor-no-num}
The methods used in this paper have the drawback that cannot be extended, at least in an obvious way, to moduli of curves of higher genus. Indeed, we exploit the fact that, for $3\leq g\leq 5$, the canononical model of a sufficiently general smooth curve of genus $g$ is a complete intersection in $\PP^{g-1}$.

This allows us to reduce the computation of $\Pic(\Mcal_g)$ to the computation of the Picard group of certain moduli stacks of smooth complete intersections, which we present as quotient stacks: the machinery of equivariant intersection theory is then applied to these stacks.
\subsection*{Structure of the paper}
In the first part of the paper, we focus on moduli stacks parametrizing smooth complete intersections. Some of them are birational to $\Mcal_g$ for $g=3,4,5$ and have the nice feature of being quite manageable from a computational point of view.

In Section \ref{sec:Fab} we introduce the moduli stack $\Fcal_{a,b}^n$ of smooth complete intersections in $\PP^n$ of codimension two and bidegree $(a,b)$, where $0<a<b$ and $n>2$. In \propref{prop:Fcal quot} we give a presentation of this stack as a quotient stack. Next, we use techniques of equivariant intersection theory to obtain a certain number of relations that hold in the integral Chow ring of $\Fcal_{a,b}^n$ (see \propref{prop:relation} and \rmkref{rmk:relations Fab}). This enables us to completely determine, in terms of generators and relations, the abelian group $\Pic(\Fcal_{a,b}^n)$ (see \thmref{thm:Pic(Fa,b)}).

In Section \ref{sec:Gdm} we move to the moduli stack $\Gcal_{d,m}^n$ of smooth complete intersections in $\PP^n$ of $m$ hypersurfaces of degree $d$, where $0<m<n$ and $n>1$. In \propref{prop:Gdmn quot} we give a presentation of this stack as a quotient stack and, just as in the previous section, we obtain a certain number of relations that hold in its integral Chow ring (see \propref{prop:Xdmn c_e} and \rmkref{rmk:relations Gdmn}), thus determining $\Pic(\Gcal_{d,m}^n)$.

Let us observe that \thmref{thm:Pic(Fa,b)} and \thmref{thm:pic Gdmn} can also be deduced from \cite{Ben}. Nevertheless, we preferred to give an independent treatment using different methods, which seem to us to provide simpler and shorter proofs. These techniques have also the advantage of being independent of the characteristic of the base field. On the other hand, we only recover some particular cases of \cite{Ben}, where also more detailed constructions are provided.

In Section \ref{sec:mg} we apply the results obtained in the previous sections in order to prove the main theorems of the paper, which are \thmref{thm:pic M3}, \thmref{thm:pic M4} and \thmref{thm:pic M5}. Along the way we also deduce some interesting corollaries, in particular \corref{cor:H3}, \corref{cor:M4ev} and \corref{cor:T5}.

\subsection*{Acknowledgements} We thank Angelo Vistoli for his constant support and Roberto Fringuelli for suggesting the problem of determining the Picard group of moduli of curves in positive characteristic. We also thank Shamil Asgarli and Giovanni Inchiostro for stimulating discussions on related arguments.
\section{Picard group of moduli of smooth complete intersections in $\PP^n$ of codimension two.} \label{sec:Fab}
In this section, we introduce and study the moduli stack of smooth complete intersections in $\PP^n$ of codimension two and bidegree $(a,b)$, with $a<b$. We denote this stack as $\Fcal_{a,b}^n$ and we define it in \parref{par:def Fab}: in \propref{prop:Fcal quot} we give a presentation of $\Fcal_{a,b}^n$ as a quotient stack, and in \thmref{thm:Pic(Fa,b)} we compute its Picard group. As observed in \rmkref{rmk:relations Fab}, we actually compute a set of other relations that hold true in the Chow ring of $\Fcal_{a,b}^n$.
\subsection{The moduli stack $\Fcal_{a,b}^n$.}
~\\
\mypar \label{par:def Fab} Fix three integers $a$, $b$ and $n$ such that $0<a<b$ and $n>2$ . Let $\Fcal_{a,b}^n$ be the category fibred in  groupoids over the site of schemes whose objects over a scheme $S$ are pairs $(V\to S,X\subset\PP(V))$, where
\begin{itemize}
	\item $V$ is a vector bundle over $S$ of rank $n+1$.
	\item $X$ is a closed subscheme of $\PP(V)$ of codimension $2$, smooth over $S$.
	\item For every geometric point $s$ in $S$ the fiber $X_s$ is a global complete intersection of bidegree $(a,b)$.
\end{itemize}
The morphisms in $\Fcal_{a,b}^n(S)$ between pairs $(V\to S,X\subset\PP(V))$ and $(V'\to S,X'\subset\PP(V'))$ are given by isomorphisms of the vector bundles $V$ and $V'$ which induce isomorphisms of $X$ and $X'$. It is easy to see that $\Fcal_{a,b}^n$ is a stack over the site of schemes.

\mypar \label{par:varphi} Let $E_{n+1}$ be the standard representation of $\GL_{n+1}$ and set \[W_a:=\Sym^a(E_{n+1}^\vee)\]
The projective space $\PP(W_a)$ is isomorphic to the Hilbert scheme of hypersurfaces of degree $a$ in $\PP^n$. In particular, there exists a universal hypersurface $\pi: H_a\arr \PP(W_a)$. This hypersurface is embedded in $\PP(W_a)\times\PP^n$, thus we can consider the restriction to $H_a$ of the invertible sheaf $\pr_2^*\Ocal_{\PP^n}(b)$, which we denote $\Ocal_{H_a}(b)$. Define $V_{a,b}:=\pi_*\Ocal_{H_a}(b)$: by cohomology and base change theorem, we have that $V_{a,b}$ is a locally free sheaf whose fibre over a closed point $[H]$ of $\PP(W_a)$ is the vector space $H^0(H,\Ocal_H(b))$.

Another way to look at $V_{a,b}$ is the following. Consider the morphism of $\GL_{n+1}$-representations:
\[ \varphi: W_a\otimes W_{b-a} \arr W_a\otimes W_b,\hspace{5pt} (f,g)\longmapsto (f,fg) \]
This induces an injective morphism of locally free sheaves over $\PP(W_a)$:
\[ W_{b-a}\otimes\Ocal_{\PP(W_a)}(-1)\arr W_b\otimes\Ocal_{\PP(W_a)}\]
whose cokernel is $V_{a,b}$. This locally free sheaf naturally inherits a $\GL_{n+1}$-linearization. 

\mypar \label{par:def X} Consider the two families of hypersurfaces $H_a\times\PP(W_b)$ and $\PP(W_a)\times H_b$ over $\PP(W_a)\times\PP(W_b)$ and denote $X_{a,b}$ their schematic intersection. There exists an open subscheme $U'_{a,b}\subset\PP(W_a)\times\PP(W_b)$ such that all the geometric fibres of the relative scheme $X_{a,b}|_{U'_{a,b}}\arr U'_{a,b}$ are global complete intersections of bidegree $(a,b)$.
Indeed, observe that the morphism $\varphi$ of \parref{par:varphi} induces an embedding of projective bundles:
\[\overline{\varphi}:\PP(W_a)\times\PP(W_{b-a})\into \PP(W_a)\times\PP(W_b) \]
We see that the locus in $\PP(W_a)\times\PP(W_b)$ where the fibres of the schematic intersection of $H_a\times\PP(W_b)$ and $\PP(W_a)\times H_b$ have codimension one is precisely $\im(\overline{\varphi})$.

\mypar\label{par:projection morphism} 
The induced projection morphism $\PP(W_a)\times\PP(W_b)\setminus\im(\overline{\varphi})\arr \PP(V_{a,b})$ makes the first scheme into an affine bundle over the second one. The restriction of the relative scheme $X_{a,b}$ to $\PP(W_a)\times\PP(W_b)\setminus\im(\overline{\varphi})$ descends along the projection to a relative scheme $\pi:C_{a,b}\arr \PP(V_{a,b})$. The fibres of this morphism
are by construction global complete intersections of bidegree $(a,b)$. 

\mypar\label{par:D} Let $D_{a,b}$ denotes the closed subscheme of $\PP(V_{a,b})$ where the fibres of the family of complete intersections $\pi:C_{a,b}\arr \PP(V_{a,b})$ are not smooth. The subscheme $D_{a,b}$ is reduced, irreducible and has codimension one. This can be seen as follows: consider the maximal open subscheme of $C_{a,b}$ where the sheaf of relative differentials $\Omega^1_{\pi}$ has rank $n-2$, and let $C^{\rm sing}_{a,b}$ be its complement. In other terms $C_{a,b}^{\rm sing}$ is the degeneracy locus of the sheaf $\Omega_{\pi}^1$.

Consider the  projection of $C_{a,b}^{\rm sing}$ onto $\PP^n$: the fibre $C_{a,b}^{\rm sing}(p)$ of $C_{a,b}^{\rm sing}$ on a geometric point $p$ of $\PP^n$ corresponds to the scheme parametrising complete intersections that have a singularity in $p$.

We claim that $C_{a,b}^{\rm sing}(p)$ is reduced and irreducible. It is enough to show that this holds for the fibre over $p_0=[1:0:\ldots:0]$, as all the fibres are isomorphic. Recall from \parref{par:projection morphism} that $\PP(W_a)\times\PP(W_b)\setminus\im(\overline{\varphi})$ is an affine bundle over $\PP(V_{a,b})$, hence we can equivalently show that the closure of the preimage of $C_{a,b}^{\rm sing}(p_0)\subset\PP(V_{a,b})$ in $\PP(W_a)\times\PP(W_b)$ is reduced and irreducible. 

This latter scheme is easy to describe. Let $[c_0:\ldots:c_n:\ldots]$ denote the point in $\PP(W_a)$ corresponding to the form $\sum_{i=0}^{n} c_iX_0^{a-1}X_i + \cdots$, and let $[d_0:\ldots:d_n:\ldots]$ denote the point in $\PP(W_b)$ associated to the form $\sum_{i=0}^{n} d_iX_0^{b-1}X_i+\cdots$ (the other terms are not relevant for our discussion).

The preimage of $C_{a,b}^{\rm sing}(p_0)$ can be described as the locus in $\PP(W_a)\times\PP(W_b)$ where
\[ c_0=0,\quad d_0=0,\quad {\rm rk}\begin{pmatrix*} c_1 & d_1 \\ \vdots & \vdots \\ c_n & d_n \end{pmatrix*}\leq 1 \Longleftrightarrow c_id_j-c_jd_i=0 .\]
It is straightforward to verify that these equations define an irreducible and reduced subscheme of $\PP(W_a)\times\PP(W_b)$.

Hence, we have proved that the fibres of $C_{a,b}^{\rm sing}\arr \PP^n$ are reduced and irreducible: as $\PP^n$ is irreducible and the morphism is proper and surjective, we can conclude that $C_{a,b}^{\rm sing}$ is irreducible and reduced. Moreover, this also shows that the fibres of $C_{a,b}^{\rm sing}\arr \PP^n$ have codimension $n+1$ in $\PP(V_{a,b})$, thus the codimension of $C_{a,b}^{\rm sing}$ in $\PP(V_{a,b})\times\PP^n$ is $n+1$: combining these observation with the fact that the morphism $C_{a,b}^{\rm sing}\arr D_{a,b}$ is generically finite, we deduce that $D_{a,b}$ is irreducible, reduced and has codimension one.

\begin{prop}\label{prop:Fcal quot}
	We have $\Fcal_{a,b}^n\simeq \left[\left(\PP(V_{a,b})\setminus D_{a,b}\right)/\GL_{n+1}\right]$.
\end{prop}
\begin{proof}
	Consider the stack $\wt{\Fcal}_{a,b}^n$ whose objects over a scheme $S$ are triples $(V\arr S,X\subset\PP(V),\alpha)$ where $(V\arr S,X\subset\PP(V))$ is an object of $\Fcal_{a,b}^n$ and $\alpha$ is an isomorphism between $V$ and $\A_S^{n+1}$. The obvious forgetful functor makes $\wt{\Fcal}_{a,b}^n$ into a $\GL_{n+1}$-torsor over $\Fcal_{a,b}^n$. We want to show that $\wt{\Fcal}_{a,b}^n\simeq \PP(V_{a,b})\setminus D_{a,b}$.
	
	The stack $\wt{\Fcal}_{a,b}^n$ is equivalent to the stack whose objects are subschemes $X\subset \PP^n_S$ such that the projection $X\to S$ is smooth and for every geometric point $s$ in $S$ the fibre $X_s\subset \PP^n_{k(s)}$ is a global complete intersection of bidegree $(a,b)$. The family
	\[C_{a,b}|_{\PP(V_{a,b})\setminus D_{a,b}}\arr \PP(V_{a,b})\setminus D_{a,b}\]
	induces a morphism $\PP(V_{a,b})\setminus D_{a,b}\arr \wt{\Fcal}_{a,b}^n$. 
	
	To construct its inverse, observe that given an object $(X\subset\PP^n_T)$ of $\wt{\Fcal}_{a,b}^n(T)$, there is a unique hypersurface $H_a$ that contains $X$: by cohomology and base change, we see that $R^1\pr_{1*}(\Ical_X\otimes\pr_2^*\Ocal(a))=0$, where $\Ical_X$ is the sheaf of ideals of $X$. Therefore, the following sequence of $\Ocal_T$-modules is exact:
	\[ 0\arr \pr_{1*}(\Ical_X \otimes \pr_2^*\Ocal(a) \arr \pr_{1*}\pr_2^*\Ocal(a) \arr \pr_{1*}(\Ocal_X\otimes\pr_2^*\Ocal(a)) \arr 0 \]
	and moreover $\pr_{1*}(\Ical_X \otimes \pr_2^*\Ocal(a))$ is an invertible sheaf. After possibly passing to a cover of $T$, we can then trivialize $\pr_{1*}(\Ical_X \otimes \pr_2^*\Ocal(a))$, obtaining in this way an injective morphism $\Ocal_T\arr W_a\otimes\Ocal_{\PP^n_T}$: the corresponding hypersurface $H_a$ obviously contains $X$, and its uniqueness follows from simple considerations on the bidegree.
	
	We have thus defined a morphism $f:T\arr \PP(W_a)$. By construction there is a well defined injective morphism:
	\[\Ical_X(b)|_{H_a}\arr \Ocal_{H_a}(b)\]
	By functoriality of the pushforward we get an injective morphism:
	\[\Lcal:=\pr_{1*}(\Ical_X(b)|_{H_a})\arr \pr_{1*}\Ocal_{H_a}(b)\]
	Observe that the sheaf on the right is isomorphic to $f^*V_{a,b}$ and the sheaf on the left is actually an invertible sheaf by the usual arguments involving cohomology and base change theorem. Thus the morphism above yields a morphism $T\arr \PP(V_{a,b})$. As everything is functorial, we have constructed a morphism $\wt{\Fcal}_{a,b}^n\arr \PP(V_{a,b})$. Moreover, the hypotheses on the objects of $X$ assure us that this morphism factorizes through $\PP(V_{a,b})\setminus D_{a,b}$. It is easy to see that the two morphisms that we have constructed are one the inverse of the other.
\end{proof}
\subsection{The Picard group of $\Fcal_{a,b}^n$}
~\\\\
\noindent From \cite{EG} we know that the Picard group of a quotient stack $[X/G]$ is equal to the $G$-equivariant Picard group of $X$. Therefore \propref{prop:Fcal quot} implies the following corollary:

\begin{cor}\label{cor:pic}
	$\Pic(\Fcal_{a,b}^n)\simeq \Pic^{\GL_{n+1}}(\PP(V_{a,b})\setminus D_{a,b})$.
\end{cor}
\mypar \label{par:equivariant picard} It is well known that $\Pic^{\GL_{n+1}}(\PP(V_{a,b}))$ is a free abelian group on three generators, namely the first Chern class $c_1$ of the standard representation of $\GL_{n+1}$, the hyperplane section $u$ of $\PP(W_a)$ and the hyperplane section $v$ of $\PP(V_{a,b})$ regarded as a projective bundle over $\PP(W_a)$. So we get:
\[ \Pic^{\GL_{n+1}}(\PP(V_{a,b})\setminus D_{a,b})\simeq \ZZ\langle c_1,u,v\rangle/\langle [D_{a,b}] \rangle \]
where $[D_{a,b}]$ denotes the equivariant cycle class of $D_{a,b}$. In the following, we will indicate as $\Ocal_{\PP(V_{a,b})}(-1)$ the tautological sheaf of $\PP(V_{a,b})$, and $\Ocal_{\PP(W_a)}(-1)$ the pullback to $\PP(V_{a,b})$ of the tautological sheaf of $\PP(W_a)$.


Let $\pi:C_{a,b}\arr \PP(V_{a,b})$ be as in \parref{par:projection morphism}, and take $C_{a,b}^{\rm sing}$ to be the degeneracy locus of the sheaf of relative differentials $\Omega^1_\pi$: it follows from the local description of the singularity that is obtained in \cite{SGA}*{Exp. XV, Th. 1.2.6} that the $(n-3)^{\rm th}$ Fitting ideal of $\Omega^1_{\pi}$ is radical, thus $C_{a,b}^{\rm sing}$ is reduced. Alternatively, to prove this claim, one can use the arguments of \parref{par:D}.

The closed subscheme $C_{a,b}^{\rm sing}$ is $\GL_{n+1}$-invariant and it is birational to $\pi(C_{a,b}^{\rm sing})$, which is equal to $D_{a,b}$. This shows that, in order to compute $[D_{a,b}]$, we can equivalently compute the cycle class $[C_{a,b}^{\rm sing}]$ in the equivariant Chow ring of $C_{a,b}$ and then take its pushforward along the projection on $\PP(V_{a,b})$.

\begin{lm}\label{lm:product}
	Let $\pr_1:\PP(V_{a,b})\times\PP^n\arr \PP(V_{a,b})$ be the projection morphism. Set $\Ecal_{a,b}:=(\pr_1^*\Ocal_{\PP(W_a)}(-1)\otimes\pr_2^*\Ocal_{\PP^n}(-a))\oplus(\pr_1^*\Ocal_{\PP(V_{a,b})}(-1)\otimes\pr_2^*(\Ocal_{\PP^n}(-b))$. Then we have:
	\[\pi_*[C_{a,b}^{\rm sing}]=\pr_{1*}([C_{a,b}]\cdot c_{n-1}^{\GL_{n+1}}([\pr_2^*\Omega^1_{\PP^n}]-[\Ecal_{a,b}]))\]
\end{lm}
\begin{proof} Consider the exact sequence:
\[ \pi^*\Omega^1_{\PP(V_{a,b})}\xrightarrow{\phi} \Omega^1_{C_{a,b}} \arr \Omega^1_{\pi} \arr 0 \]
The morphism $\phi$ has generically rank $d=\dim(\PP(V_{a,b}))$, and the sequence above implies that the degeneracy locus $\bD_{d-1}(\phi)$ coincides with $C_{a,b}^{\rm sing}$. Observe that the codimension of $C_{a,b}^{\rm sing}$ is equal to the expected codimension of $\bD_{d-1}(\phi)$. We can apply the equivariant version of Thom-Porteous formula (see \cite{Ful}*{Sec. 14.4}) which tells us that:
\[ [C_{a,b}^{\rm sing}]=[\bD_{d-1}(\phi)]=c_{n-1}^{\GL_{n+1}}([\Omega^1_{C_{a,b}}]-[\pi^*\Omega^1_{\PP(V_{a,b})}]) \]
If $i:C_{a,b}\hookrightarrow \PP(V_{a,b})\times\PP^n$ denotes the closed embedding, which is regular, then we have:
\[0\arr i^*\Ical_{C_{a,b}}\arr i^*\Omega^1_{\PP(V_{a,b})\times\PP^n}\arr \Omega^1_{C_{a,b}}\arr 0\]
From this we see that:
\[ [\Omega^1_{C_{a,b}}]=i^*([\Omega^1_{\PP(V_{a,b})\times\PP^n}]-[\Ical_{C_{a,b}}]) \]
inside $K_0^{\GL_{n+1}}(C_{a,b})$. We also know by construction that $i^*\Ical_{C_{a,b}}=i^*\Ecal_{a,b}$.

We have the obvious identity:
\[ [\pi^*\Omega^1_{\PP(V_{a,b})}]=i^*[\pr_1^*\Omega^1_{\PP(V_{a,b})}] \]
Recall that $\Omega^1_{\PP(V_{a,b})\times\PP^n}\simeq \pr_1^*\Omega^1_{\PP(V_{a,b})}\oplus\pr_2^*\Omega^1_{\PP^n}$. Consequently:
\[ [\Omega^1_{\PP(V_{a,b})\times\PP^n}]-[\pr_1^*\Omega^1_{\PP(V_{a,b})}]=[\pr_2^*\Omega^1_{\PP^n}] \]
Putting all together, we deduce:
\begin{align*}
\pi_*c_{n-1}^{\GL_{n+1}}([\Omega^1_{C_{a,b}}]-[\pi^*\Omega^1_{\PP(V_{a,b})}])&=\pi_*c_{n-1}^{\GL_{n+1}}(i^*[\pr_2^*\Omega^1_{\PP^n}]-i^*[\Ical_{C_{a,b}}])\\
&=\pr_{1*}i_*i^*c_{n-1}^{\GL_{n+1}}([\pr_2^*\Omega^1_{\PP^n}]-[\Ecal_{a,b}])\\
&=\pr_{1*}([C_{a,b}]\cdot c_{n-1}^{\GL_{n+1}}([\pr_2^*\Omega^1_{\PP^n}]-[\Ecal_{a,b}]))
\end{align*}
and we are done.
\end{proof}

\begin{lm}\label{lm:class Xa,b}
	Let $u$ (resp. $v$, $t$) be the pullback to $U_{a,b}$ of the hyperplane section of $\PP(W_a)$ (resp. $\PP(V_{a,b})$, $\PP^n$). Then $[C_{a,b}]=(u+at)(v+bt)$.	
\end{lm}
\begin{proof}
	Observe that $C_{a,b}$ is the complete intersections of two independent global sections of the locally free sheaf \[(\pr_1^*\Ocal_{\PP(W_a)}(1)\otimes \pr_2^*\Ocal_{\PP^n}(a))\oplus(\pr_1^*\Ocal_{\PP(V_{a,b})}(1)\otimes \pr_2^*\Ocal_{\PP^n}(b))\]
	This implies that $[C_{a,b}]$ is equal to the top Chern class of the locally free sheaf above. After a straightforward computation, we get the desired conclusion.
\end{proof}
\begin{lm}\label{lm:c2}
	Set $\Ecal_{a,b}$ as in \lmref{lm:product} and let $\ell_1,\ldots,\ell_{n+1}$ be the Chern roots of the standard $\GL_{n+1}$-representation. Then:
	\begin{align*}
	c_{n-1}^{\GL_{n+1}}([\pr_2^*\Omega^1_{\PP^n}]-[\Ecal_{a,b}])&=\left\{\prod_{i=1}^{n+1}(1-\ell_i-t) \cdot \sum_{j=0}^{\infty} a^jt^j \cdot \sum_{k=0}^{\infty} b^kt^k\right\}_{n-1}
	\end{align*}
\end{lm}
\begin{proof}
	Let $E_{n+1}$ be the standard representation of $\GL_{n+1}$. Then the Euler exact sequence for $\PP^n=\PP(E_{n+1})$ is:
	\[ 0\arr \Omega^1_{\PP^n}\arr E_{n+1}^{\vee}\otimes\Ocal(-1))\arr \Ocal \arr 0 \]
    Thus 
    \[c^{\GL_{n+1}}(\Omega^1_{\PP^n})=c^{\GL_{n+1}}(E_{n+1}^\vee\otimes\Ocal(-1))=\prod_{i=1}^{n+1}(1-\ell_i-t)\]
    where $\ell_1,\ldots,\ell_{n+1}$ are the Chern roots of $E_{n+1}$: as the expression above is symmetric in the $\ell_i$ it would be possible to rewrite it in terms of the usual Chern classes $c_1,\ldots,c_{n+1}$. On the other hand, the form above is more suitable for doing computations.
    
	We also have:
	\[ c^{\GL_{n+1}}(\Ecal_{a,b})=(1-(u+at))(1-(v+bt)) \]
	From this we deduce:
	\begin{align*}
	&c^{\GL_{n+1}}([\pr_2^*\Omega^1_{\PP^n}]-[\Ecal_{a,b}])=\pr^*_2c^{\GL_{n+1}}(\Omega^1_{\PP^n})\cdot c^{\GL_{n+1}}(\Ecal_{a,b})^{-1}\\&= c^{\GL_{n+1}}(\Omega^1_{\PP^n})\cdot (1+(u+at)+(u+at)^2+\dots)\cdot (1+(v+at)+(v+at)^2+\dots) 
	\end{align*}
	Expanding the expression above and taking the degree $n-1$ part, we get the desired conclusion.
\end{proof}	

\begin{prop}\label{prop:relation}
	We have:
	\begin{align*}
	\pi_*[C^{\rm sing}_{a,b}]=&\left( \sum (-1)^k\binom{n}{k}a^{i+1}b^{j+1} - \sum (-1)^k\binom{n+1}{k}a^ib^j \right)c_1\\
	&+\left( \sum (-1)^k\binom{n+1}{k}a^ib^{j+1} (1+(i+1)a)  \right)u\\
	&+\left( \sum (-1)^k\binom{n+1}{k}a^{i+1}b^j (1+(j+1)b) \right)v
	\end{align*}
\end{prop}
\begin{proof}
	Recall from \parref{par:D} that $C_{a,b}^{\rm sing}$ has codimension $n+1$ in $\PP(V_{a,b})\times\PP^n$, hence there is a unique way of writing it as
	\[ i_*[C_{a,b}^{\rm sing}]= \xi_1 t^n + \xi_2 t^{n-1} +\cdots + \xi_{n+1} \]
	Then $\pi_*[C_{a,b}^{\rm sing}]=\pr_{1*}i_*[C_{a,b}^{\rm sing}]=\xi_1$. 
	
	 \lmref{lm:class Xa,b} and \lmref{lm:c2} give us an almost explicit expression for $i_*[C_{a,b}^{\rm sing}]$ in the equivariant Chow ring of $\PP(V_{a,b})\times\PP^n$, which is the following:
	 \begin{equation}\label{eq:1}
	 i_*[C_{a,b}^{\rm sing}] = (u+at)(v+bt)\cdot F(u,v,t,\underline{\ell})
	 \end{equation}
	 where
	 \begin{equation}\label{eq:2}
	  F(u,v,t,\underline{\ell}) := \left\{\prod_{i=1}^{n+1}(1-\ell_i-t) \cdot \sum_{j=0}^{\infty} a^jt^j \cdot \sum_{k=0}^{\infty} b^kt^k\right\}_{n-1}
	 \end{equation}
	 Let $F_{n-1-i}(u,v,\underline{\ell})$ denote the coefficient of $t^i$ in $F(u,v,\underline{\ell})$. We can use the relation
	 \[ t^{n+1} = -(c_{n+1}+c_{n}t+c_{n-1}t^2+c_1t^n)\]
	 to put (\ref{eq:1}) in its canonical form, from which we see that the coefficient in front of $t^n$ is
	 \[	 \xi_1 = (av+bu)F_0 + ab (F_1 - c_1F_0) \]
	 An explicit expression for $F_0$ can be obtained by evaluating (\ref{eq:2}) in $u=v=\ell_i=0$. We get:
	 \begin{align*}
	 F_0 &= \left\{ (1-t)^{n+1} \cdot \sum_{i=0}^{\infty} a^it^i \cdot \sum_{j=0}^{\infty} b^j t^j \right\}_{n-1} = \sum (-1)^k\binom{n+1}{k}a^ib^j 
	 \end{align*} 
	 where the last sum is taken over the triples $(i,j,k)$ such that $i,j,k\geq 0$ and $i+j+k=n-1$.
	 
	 Also $F_1$ can be computed with almost the same trick, because:
	 \[ F_1 = (\partial_u F)(0) u  + (\partial_v F)(0)v + (\partial_{\ell_1} F)(0)\left(\sum \ell_i\right) \]
	 After some computation and after having substitued $\sum \ell_i=c_1$, we get:
	 \begin{align*}
	 F_1 &= \left( \sum (-1)^k\binom{n+1}{k}(i+1)a^ib^j  \right)u \\
	 &+ \left( \sum (-1)^k\binom{n+1}{k}(j+1)a^ib^j \right)v \\
	 &+\left( \sum (-1)^k\binom{n}{k}a^ib^j\right)c_1
	 \end{align*}
	 where the sums are taken over the triples $(i,j,k)$ such that $i,j,k\geq 0$ and $i+j+k=n-1$. Putting all together, we obtain an explicit espression for $\xi_1=\pi_*[C_{a,b}^{\rm sing}]$.
\end{proof}
Let $r_{a,b,n}(c_1,u,v)$ be the expression appearing in \propref{prop:relation}. We are ready to state and prove the main result of the section:

\begin{thm}\label{thm:Pic(Fa,b)}
	$\Pic(\Fcal_{a,b}^n)=\ZZ\langle c_1,u,v \rangle/\langle r_{a,b,n}(c_1,u,v) \rangle$
\end{thm} 
\begin{proof}
	\corref{cor:pic} tells us that $\Pic(\Fcal_{a,b}^n)=\Pic^{\GL_{n+1}}(\PP(V_{a,b})\setminus D_{a,b})$. We already observed in \parref{par:equivariant picard} that the group on the right is generated by the elements $c_1$, $u$ and $v$, with a unique relation given by the cycle class $[D_{a,b}]$. In \parref{par:projection morphism} we reduced the computation of $[D_{a,b}]$ to the computation of $\pi_*[C_{a,b}^{\rm sing}]$. Then \propref{prop:relation} permits us to conclude.
\end{proof}

\begin{rmk}\label{rmk:relations Fab}
    As already observed in the proof of \propref{prop:relation}, there is a unique way to express $[C_{a,b}^{\rm sing}]$ as a polynomial in $t$ of degree $n$. Let $\xi_i$ be the coefficients of this polynomial in $t$. These cycles can be seen as polynomials in the generators of the $\GL_{n+1}$-equivariant Chow ring of $\PP(V_{a,b})$, and can be explicitly computed using \propref{prop:relation}.
    
    Then $\xi_1,\ldots,\xi_n$ gives relations in the Chow ring of $\Fcal_{a,b}^n$: we may ask ourselves if these relations actually generate the whole ideal of relations of this integral Chow ring.
\end{rmk}
\section{Picard group of moduli of smooth equidegree complete intersections in $\PP^n$}\label{sec:Gdm}
In this section we introduce and study the moduli stack $\Gcal_{d,m}^n$ of complete intersections of $m$ hypersurfaces of degree $d$ in $\PP^n$. This stack is defined in \parref{par:Gdmn}: in \propref{prop:Gdmn quot} we present $\Gcal_{d,m}^n$ as a quotient stack, and in \thmref{thm:pic Gdmn} we compute its Picard group. As observed in \rmkref{rmk:relations Gdmn}, we actually compute a set of other relations that hold true in the Chow ring of $\Gcal_{d,m}^n$.
\subsection{The moduli stack $\Gcal_{d,m}^n$.}
~\\
\mypar \label{par:Gdmn} Fix three integers $d$, $m$ and $n$ such that $d>0$ and $0<m<n$.
Let $\Gcal_{d,m}^n$ be the category fibred in groupoids over the site of schemes whose objects over a scheme $S$ are pairs $(V\arr S, X\subset\PP(V))$, where:
\begin{itemize}
	\item $V$ is a vector bundle over $S$ of rank $n+1$.
	\item $X$ is a closed subscheme of $\PP(V)$ of codimension $m$, smooth over $S$.
	\item For every geometric point $s$ in $S$ the fiber $X_s$ is a global complete intersection of $m$ hypersurfaces of degree $d$.
\end{itemize}
The morphisms in $\Gcal_{d,m}^n(S)$ between pairs $(V\to S,X\subset\PP(V))$ and $(V'\to S,X'\subset\PP(V'))$ are given by isomorphisms of the vector bundles $V$ and $V'$ which induce isomorphisms of $X$ and $X'$. It is easy to see that $\Gcal_{d,m}^n$ is a stack over the site of schemes.

\mypar\label{par:Xdmn} Let $E_{n+1}$ be the standard representation of $\GL_{n+1}$ and define $W_{d}:=\Sym^d E_{n+1}^{\vee}$. Let $\Gras$ be the grassmannian of $m$-dimensional subspaces of $W_{d}$. It naturally inherits a $\GL_{n+1}$-action. We denote $\Tcal\dmn$ the universal, locally free subsheaf of $W_{d}\otimes\Ocal_{\Gras}$ of rank $m$.

Consider the product $\Gras\times\PP^n$ with the two projections $\pr_1$ and $\pr_2$. Then we have a surjective morphism:
\[ W_{d}\otimes\Ocal_{\Gras\times\PP^n}=\pr_2^*(W_{d}\otimes\Ocal_{\PP^n}) \arr \pr_2^*\Ocal_{\PP^n}(d) \]
because $\Ocal_{\PP^n}(d)$ it is globally generated. This induces:
\[ \pr_1^*\Tcal\dmn\otimes\pr_2^*\Ocal_{\PP^n}(-d)\arr W_{d}\otimes\pr_2^*\Ocal_{\PP^n}(-d) \arr \pr_2^*\Ocal_{\PP^n}=\Ocal_{\Gras\times\PP^n}\]
Let $\Ical_{X'\dmn}$ to be the sheaf of ideals generated by the image of the morphism above, and denote $X'\dmn\subset \Gras\times\PP^n$ the closed subscheme defined by $\Ical_{X'\dmn}$. We have an obvious morphism $X'\dmn\arr \Gras$, which is proper.

There exists an open $\GL_{n+1}$-invariant subscheme $U\dmn$ of $\Gras$, whose complement has codimension greater than one, such that the fibres of the restricted morphism 
\[\pi:X'\dmn|_{U\dmn}=:X\dmn\arr U\dmn\]
are complete intersections of $m$ hypersurfaces of degree $d$. This open subscheme actually coincides with the subscheme of $\Gras$ where the fibres of $X'\dmn$ are local complete intersections, which is well known to be open.

Moreover, there exists a $\GL_{n+1}$-invariant divisor $D\dmn$ inside $U\dmn$ such that the scheme $X\dmn$ restricted over $U\dmn\setminus D\dmn$ is smooth. This can be proved using the same arguments of \parref{par:D}.

\begin{prop}\label{prop:Gdmn quot}
	We have $\Gcal_{d,m}^n\simeq \left[\left(U\dmn\setminus D\dmn\right)/\GL_{n+1}\right]$.
\end{prop}
\begin{proof}
	Define $\wt{\Gcal}\dmn^{n}$ as the category fibred in groupoids over the site of schemes whose objects are triples $(V\arr S, X\subset\PP(V), \alpha)$ such that the pair $(V\arr S, X\subset\PP(V))$ is an object of $\Gcal_{d,m}^n(S)$ and $\alpha$ is an isomorphism between $V$ and $\A^{n+1}_S$. The forgetful functor $\wt{\Gcal}\dmn^{n}\arr \Gcal_{d,m}^n$ is a $\GL_{n+1}$-torsor, where $\GL_{n+1}(S)$ acts on the isomorphism $\alpha$ in the obvious way. We want to show that $\wt{\Gcal}\dmn^{n}$ is isomorphic to $U\dmn\setminus D\dmn$. Observe that the stack $\wt{\Gcal}\dmn^{n}$ is equivalent to the stack whose objects are subschemes $X$ of $\PP^n_S$, smooth over $S$, whose geometric fibres are global complete intersections of $m$ hypersurfaces of degree $d$ in $\PP^n_{k(s)}$. 
	
	The construction of $X\dmn$ of \parref{par:Xdmn} yields a morphism $U\dmn\setminus D\dmn\arr \wt{\Gcal}\dmn^{n}$. The inverse morphism is defined as follows: given a scheme $S$, let $X\subset\PP^n_S$ be a closed subscheme, smooth over $S$, whose fibres are global complete intersections of $m$ hypersurfaces of degree $d$. Let $\pr_1$ (resp. $\pr_2$) denote the projection morphism on $S$ (resp. on $\PP^n$). Let $\Ical_X$ be the sheaf of ideals of $X$, so that we have the inclusion $\Ical_X \hookrightarrow \Ocal_{\PP^n_S}$.
	
	We can tensorize this inclusion with $\pr_2^*\Ocal(d)$ and then take its pushforward along the first projection, so to get an inclusion:
	\[ \pr_{1*}(\Ical_X\otimes \pr_2^*\Ocal(d))\hookrightarrow \pr_{1*}\pr_2^*\Ocal(d)=W_{d}\otimes\Ocal_S \]
	The sheaf on the left is locally free of rank $m$: this can be easily proved using cohomology and base change. We have constructed in this way a morphism $\wt{\Gcal}\dmn^{n}(S)\arr (U\dmn\setminus D\dmn)(S)$ for every scheme $S$. As everything is functorial, we actually get the desired morphism $\wt{\Gcal}\dmn^n\arr U\dmn\setminus D\dmn$, which can be easily checked to be the inverse of the morphism $U\dmn\setminus D\dmn \arr \wt{\Gcal}\dmn^n$ that we defined before.
\end{proof}

\subsection{The Picard group of $\Gcal_{d,m}^n$}
~\\\\
\noindent As before, from \propref{prop:Fcal quot} we deduce the following corollary:

\begin{cor}\label{cor:pic G}
	$\Pic(\Gcal_{d,m}^n)\simeq \Pic^{\GL_{n+1}}(U\dmn\setminus D\dmn)$.
\end{cor}
\mypar \label{par:equi pic} Recall that the complement of $U\dmn$ in $\Gras$ has codimension greater than one. This implies:
\[ \Pic^{\GL_{n+1}}(U\dmn)=\Pic^{\GL_{n+1}}(\Gras)\]
It is well known that the $\GL_{n+1}$-equivariant Picard group of $\Gras$ is a free abelian group on two generators $c_1$ and $s_1$. Namely, the cycle $c_1$ is the first Chern class of the standard representation $E_{n+1}$ and $s_1$ is the first special Schubert class. In particular, $s_1=-c_1(\Tcal\dmn)$, where $\Tcal\dmn$ is the tautological subsheaf on $\Gras$. In this way we obtain:
\[ \Pic^{\GL_{n+1}}(U\dmn\setminus D\dmn)=\ZZ\langle c_1,s_1 \rangle/\langle [D\dmn] \rangle \]
Therefore, we only have to compute the cycle class of $D\dmn$.

\mypar\label{par:class Ddmn} In \parref{par:Xdmn} we constructed a closed subscheme $i:X\dmn\hookrightarrow U\dmn\times\PP^n$ whose geometric fibres over $U\dmn$ are global complete intersections of $m$ hypersurfaces of degree $d$ in $\PP^n$. Let $X\dmn^{\rm sing}$ be the singular locus of $X\dmn$ over $U\dmn$: by this we mean the degeneracy locus of the sheaf of relative differentials $\Omega^1_{\pr_1}$. 
If the characteristic of the base field $k$ is different from two, or if $n-m$ is odd, then this locus, defined by the $(n-m-1)^{\rm th}$ Fitting ideal of $\Omega^1_{\pr_1}$, is generically reduced and we have:
\[ \pr_{1*}i_*[X\dmn^{\rm sing}]=[D\dmn]\]
If the characteristic of the base field is two and $n-m$ is even, then we have:
\[ \pr_{1*}[X\dmn^{\rm sing}]=\pr_{1*}(2[(X\dmn^{\rm sing})_{\rm red}])=2[D\dmn] \]
These assertions follows from the following two observations: first, the restricted morphism from $(X\dmn^{\rm sing})_{\rm red}\arr D\dmn$ is birational, as the generic fibre hase only one isolated and ordinary quadratic singularity. Second, the scheme structure of $X^{\rm sing}\dmn$, induced by the $(n-m-1)^{\rm th}$ Fitting ideal of $\Omega^1_{\pr_1}$, can be deduced from \cite{SGA}*{Exp. XV, Th. 1.2.6}: we see that, in the case that the characteristic of the base field is two and $n-m$ is even, the length of the structure sheaf of $X\dmn^{\rm sing}$ localized at its (unique) irreducible component is two, and otherwise is one.

\begin{lm}\label{lm:product 2}
	We have:
	\[ \pr_{1*}i_*[X\dmn^{\rm sing}]=\pr_{1*}([X\dmn]\cdot c_{n-m+1}^{\GL_{n+1}}([\pr_2^*\Omega^1_{\PP^n}]-[\pr_1^*\Tcal\dmn\otimes\pr_2^*\Ocal_{\PP^n}(-d)]))\]
\end{lm}
\begin{proof}
	Observe that, if we denote $i:X\dmn\hookrightarrow U\dmn\times\PP^n$ the closed embedding, we have:
	\[ i^*\Ical_{X\dmn}=i^*(\pr_1^*\Tcal\dmn\otimes\pr_2^*\Ocal_{\PP^n}(-d)) \]
	With this minor change, the proof is basically the same as the one of \lmref{lm:product}.
\end{proof}
\begin{rmk}
	More generally, suppose to have two smooth varieties $B$ and $P$ of dimension $d$ and $n$ and a smooth subscheme $i:X\hookrightarrow B\times P$ of codimension $m$. 
	Let $r$ be the generic rank of $\Omega^1_{X/B}$ and define $X_{k}$ to be the locus of $X$ where the rank of $\Omega^1_{X/B}$ is less or equal to $k$, where $0\leq k\leq r$. Using the same notation of \cite{Ful}*{Sec. 14.4} we see that the arguments of the proof of \lmref{lm:product} give the following formula:
	\[ i_*[X_k]=[X]\cdot\Delta^{(d-k)}_{(n+d-m-k)}(\pr_2^*[\Omega^1_P]-[\Ical_X]) \]
	where $\Ical_X$ is the sheaf of ideals of $X$. It naturally extends to the equivariant case.
\end{rmk}

\begin{lm}\label{lm:class Xdmn}
	Let $\Tcal\dmn$ be the tautological subsheaf as in \parref{par:Xdmn} and let $t$ denotes the pullback of the hyperplane section of $\PP^n$ to $\Gras\times\PP^n$. Then we have: $[X\dmn]=\sum_{l=0}^{m}(-1)^{m-l}d^lt^lc_{m-l}(\Tcal\dmn)$.
\end{lm}
\begin{proof}
	By construction we have that $X\dmn$ is equal to the intersection of $m$ global sections of the rank $m$ locally free sheaf $\pr_1^*\Tcal^\vee\otimes\pr_2^*\Ocal(d)$. This implies that $[X]=c_m^{\GL_{n+1}}(\pr_1^*\Tcal^\vee\otimes\pr_2^*\Ocal(d))$. Applying the formula for the top Chern class of a tensor product of a vector bundle with an invertible line bundle (see \cite{Ful}*{Sec. 3.2}), we get the desired result.
\end{proof}

\begin{lm}\label{lm:c_e}
	Let $e=n-m+1$ and set $\Ecal\dmn:=\pr_1^*\Tcal\dmn\otimes\pr_2^*\Ocal_{\PP^n}(-d)$. Then $c_e^{\GL_{n+1}}(\pr_2^*[\Omega^1_{\PP^n}]-[\Ecal\dmn])$ is equal to:
	\[\sum (-1)^i\binom{n+1-j}{i-j}\binom{n-i}{m-1+k}d^{e-i-k}c_j \pr_1^*s_k(\Tcal\dmn) t^{e-j-k} \]
	where $0\leq i\leq e$, $0\leq j\leq i$ and $0\leq k\leq e-i$, and $s_k(\Tcal\dmn)$ is the $k^{\rm th}$ Segre class.
\end{lm}
\begin{proof}
	In order not to make the formulas below notationally too heavy, we will suppress the apexes for the equivariant Chern classes. We have:
	\begin{align*}
	c_e(\pr_2^*[\Omega^1_{\PP^n}]-[\Ecal\dmn])&=\left\{ \pr_2^*c(\Omega^1_{\PP^n})s(\Ecal\dmn) \right\}_e\\
	&=\sum_{i=0}^{e} \pr_2^*c_i(\Omega^1_{\PP^n})s_{e-i}(\Ecal\dmn)
	\end{align*}	
	The Euler exact sequence implies that $c(\Omega^1_{\PP^n})=c(E^{\vee}\otimes\pr_2^*\Ocal(-1))$. From \cite{Ful}*{Sec. 3.2}, we know that:
	\[c_i(E^\vee \otimes \pr_2^*\Ocal(-1))=\sum_{j=0}^{i} (-1)^i\binom{n+1-j}{i-j}c_jt^{i-j}\]
	Applying the formula for the Segre class of a tensor product of a vector bundle with a line bundle (see \cite{Ful}*{Sec. 3.1}), we get:
	\[ s_{e-i}(\Ecal\dmn)=\sum_{k=0}^{e-i} \binom{n-i}{m-1+k}d^{e-i-k}\pr_1^*s_k(\Tcal\dmn) t^{e-i-k}  \]
	Putting all together, we get the desired expression.
\end{proof}
\begin{prop}\label{prop:Xdmn c_e}
	We have:
	\begin{align*}
	[X^{\rm sing}\dmn]=\sum & (-1)^{m-l+i}\binom{n+1-j}{i-j}\binom{n-i}{m-1+k}\cdot\\&\cdot d^{e+l-i-k} c_j \pr_1^*(s_k(\Tcal\dmn)c_{m-l}(\Tcal\dmn)) t^{e+l-j-k}
	\end{align*}
	where $0\leq i\leq e$, $0\leq j\leq i$, $0\leq k\leq e-i$ and $0\leq l\leq m$.
\end{prop}
\begin{proof}
	It easily follows from the proofs of \lmref{lm:product 2}, \lmref{lm:class Xdmn} and \lmref{lm:c_e}.
\end{proof}	
\begin{cor}\label{cor:Ddmn}
	We have:
	\begin{align*}
	[D\dmn] =& \alpha^{-1}\left(\sum_{i=0}^{n-m+1} (-1)^i\binom{n+1}{i}\binom{n+1-i}{m}d^{n-i}\right) s_1\\
	&+  \alpha^{-1}\left( \sum_{i=0}^{n-m+1} (-1)^{i+1}\binom{n}{i}\binom{n-i}{m-1} d^{n+1-i} \right) c_1
	\end{align*}
	where $\alpha$ is equal to two if the characteristic of the base field is two and $n-m$ is even, and is equal to one otherwise.
\end{cor}
\begin{proof}
	From \parref{par:class Ddmn} we know that $\pr_{1*}[X\dmn^{\rm sing}]=\alpha[D\dmn]$. Using \propref{prop:Xdmn c_e}, we can write \[ [X\dmn^{\rm sing}]=\sum_{q=0}^{n+1} \xi_qt^q\]
	where the coefficients $\xi_q$ are polynomials in $c_j$, $s_k(\Tcal\dmn)$ and $c_p(\Tcal\dmn)$. Using the relation:
	\[ \sum_{i=0}^{n+1} c_{i}t^{n+1-i} = 0 \]
	we deduce that $\pr_{1*}[X^{\rm sing}\dmn]= \xi_n-\xi_{n+1}c_1$.
	To explicitly compute the coefficient $\xi_n$, we have to look at the addends in the summation of \propref{prop:Xdmn c_e} where one and only one among the classes $c_1$, $s_1(\Tcal\dmn)$ and $c_1(\Tcal\dmn)$ appears. Then, to obtain the desired formula, we have to use the fact that $c_1(\Tcal)=-s_1(\Tcal)=:s_1$ and the Pascal identity.
\end{proof}
Let $r\dmn(c_1,s_1)$ be the expression appearing in \corref{cor:Ddmn}. Then we have:

\begin{thm}\label{thm:pic Gdmn}
	$\Pic(\Gcal_{d,m}^n)=\ZZ\langle c_1,s_1 \rangle/\langle r\dmn(c_1,s_1) \rangle$.
\end{thm}
\begin{proof}
	It follows from \corref{cor:pic G}, \parref{par:equi pic} and \corref{cor:Ddmn}.
\end{proof}
\begin{rmk}\label{rmk:relations Gdmn}
    There is a unique way to express $[X\dmn^{\rm sing}]$ as a polynomial in $t$ of degree $n$. Denote the coefficients of this polynomial as $\zeta_i$, where $0\leq i\leq n$: these cycles can be seen as polynomials in the generators of the $\GL_{n+1}$-equivariant Chow ring of $U\dmn$, and can be explicitly computed using \propref{prop:Xdmn c_e}.
    
    Then the cycles $\zeta_1,\dots,\zeta_n$ gives relations in the Chow ring of $\Gcal_{d,m}^n$: we may ask ourselves, just as we have done in \rmkref{rmk:relations Fab}, if these relations actually generate the whole ideal of relations of this Chow ring.
\end{rmk}
\section{Picard group of moduli of smooth curves of low genus} \label{sec:mg}
The results obtained so far are applied in this section in order to compute the Picard group of $\Mcal_g$, the moduli stack of smooth curves of genus $g$, for $g=3,4,5$. The main results are \thmref{thm:pic M3}, \thmref{thm:pic M4} and \thmref{thm:pic M5}. Moreover, we also easily deduce, using this machinery, an explicit expression of some geometrically meaningful divisors on these stacks in terms of the generator of the Picard group. Namely, in \corref{cor:H3} we compute the cycle class of the substack of hyperelliptic curves of genus three, in \corref{cor:M4ev} we compute the cycle class of the substack of smooth curves of genus four with an even theta characteristic and in \corref{cor:T5} we compute the cycle class of the stack of trigonal curves of genus five.

We recall here a technical result that will be frequently used throughout the section:

\begin{lm}\label{lm:base change}
	Let $\pi:C\arr S$ be a smooth and proper morphism whose fibres are non-hyperelliptic genus $g$ curves, and let $\omega_\pi$ denote the relative dualizing sheaf. Then:
	\begin{enumerate}
		\item $\pi_*\omega_\pi$ is a locally free sheaf of rank $g$ and its formation commutes with base change. 
		\item The canonical morphism $\pi^*\pi_*\omega_\pi\arr \omega_\pi$ is surjective.
	\end{enumerate} 
\end{lm}
\begin{proof}
	It follows from the cohomology and base change theorem: the proof of \cite{Ols}*{Sec. 8.4} works also for this case, after changing the tricanonical sheaf with the canonical one.
\end{proof}
\subsection{Genus three case}
~\\
\mypar \label{par:U3'} Let $\Mcal_3$ be the moduli stack of smooth curves of genus three, and let $\Hcal_3$ be the substack of hyperelliptic curves, and define $\Ucal_3:=\Mcal_3\setminus\Hcal_3$. Thanks to \lmref{lm:base change}, we can consider the category fibred in groupoids over the site of schemes whose objects are triples $(\pi,i,\beta)$ where:
\begin{itemize}
    \item $\pi:C\arr S$ is an object of $\Ucal_3$.
	\item $i:C\hookrightarrow\PP(\pi_*\pi^*\omega_\pi)$ is a closed embedding.
	\item $\beta$ is an isomorphism between $i^*\Ocal(1)$ and $\omega_{C/S}$.
\end{itemize}
There is an obvious morphism $\Ucal_3'\arr \Ucal_3$. We also have a morphism in the opposite direction: indeed, given a smooth family of genus three, non-hyperelliptic curves $\pi:C\arr S$, by \lmref{lm:base change} we have a canonical, surjective morphism $\pi^*\pi_*\omega_\pi\arr \omega_\pi$, which in turn induces a canonical embedding $i:C\hookrightarrow\PP(\pi_*\omega_\pi)$ and an isomorphism $\beta:i^*\Ocal(1)\simeq\omega_\pi$. As everything is functorial, we get the claimed morphism $\Ucal_3 \arr \Ucal'_3$. It is almost immediate to check that the two morphisms are equivalences of stacks.

\mypar \label{par:Pdmn} Let $\Gcal_{d,m}^n$ be the stack introduced in \parref{par:Gdmn}. We can consider the following invertible sheaf:
\[ \Lcal\dmn: (V\arr S,X\subset\PP(V))\longmapsto \pr_*(\det(\Ical_X|_X)\otimes\Ocal_X(mn))\otimes\det(V)\]
where $\pr:\PP(V)\arr S$ is the projection. The sheaf appearing on the right is locally free by cohomology and base change theorem. Denote as $\Pcal\dmn^n$ the associated $\GG_m$-torsor. More precisely, the objects of $\Pcal\dmn^n$ are triples $(V\arr S,X\subset\PP(V),\gamma)$, where the pair $(V\arr S,X\subset\PP(V))$ is an object of $\Gcal_{d,m}^n$ and $\gamma$ is a trivializing section of $\pr_*(\det(\Ical_X|_X)\otimes\Ocal_X(mn))\otimes\det(V)$.

\begin{prop}\label{prop:U3 iso P41}
	Set $n=2$. Then we have $\Ucal_3\simeq \Pcal_{4,1}^2$.
\end{prop}
\begin{proof}
	From \parref{par:U3'} we see that we can equivalently show that $\Ucal_3'$ is isomorphic to $\Pcal_{4,1}^2$.
	First we construct a morphism $\Ucal_3'\arr \Pcal_{4,1}^2$. Recall that the canonical model of a family $\pi:C\arr S$ of smooth, non-hyperelliptic curves of genus three is a smooth quartic, thus given an object $(\pi,i,\beta)$ of $\Ucal_3'$ the pair $(\VV(\pi_*\omega_\pi)\arr S,i(C)\subset\PP(\pi_*\omega_\pi))$ is an object of $\Gcal_{4,1}$.
	
	The isomorphism $\beta$ can be seen as a global, everywhere non-vanishing, section of $i^*\Ocal(-1)\otimes\omega_\pi$. We have the following chain of easy identifications:
	\begin{align*}
	H^0(C,i^*\Ocal(-1)\otimes\omega_\pi)&=H^0(\PP(\pi_*\omega_\pi),i_*(i^*(\Ocal(-1)\otimes \Ical_C^{\vee} \otimes\omega_{\PP(\pi_*\omega_\pi)})))\\
	&= H^0(\PP(\pi_*\omega_\pi), \Ical_C^{\vee}\otimes \omega_{\PP(\pi_*\omega_\pi)}(-1))
	\end{align*}
	Using the canonical isomorphism $\omega_{\PP(\pi_*\omega_\pi)}\simeq \Ocal(-3)\otimes\det((\pi_*\omega_\pi)^\vee)$, we deduce that the isomorphism $\beta$ can actually be regarded as a trivialization $\gamma$ of the invertible sheaf $\pi_*(i^*(\Ical_C(4)))\otimes\det(\pi_*\omega_\pi)$, so that the triple $(\VV(\pi_*\omega_\pi)\arr S,i(C)\subset\PP(\pi_*\omega_\pi),\gamma)$ is an object of $\Pcal_{4,1}^2$.
    As everything is functorially well behaved, this defines a morphism $\Ucal_3'\arr \Pcal_{4,1}^2$.
	
	To construct the inverse morphism, consider an object $(V\arr S, X\subset \PP(V),\gamma)$ of $\Pcal_{4,1}^2$: with the same argument used before, we see that the trivializing section $\gamma$ induces an isomorphism $\beta:i^*\Ocal(1)\simeq \omega_{X/S}$, that pushed forward to $S$ allows us to identify $V$ with $\pr|_{X*}\omega_{X/S}$, using the canonical isomorphism $\pr_*\Ocal(1)\simeq V$. We use this identification to define the closed embedding $i:X\subset\PP(V)\simeq\PP(\pr|_{X*}\omega_{X/S})$. Therefore, we can construct the morphism $\Pcal_{4,1}^2\arr \Ucal_3'$ by sending a triple $(V\arr S, X\subset\PP(V),\gamma)$ to the object $(\pr|_X,i,\beta)$.
	
	It is immediate to check that the two morphisms that we have defined are equivalences of stacks.
\end{proof}

\mypar \label{par:Pic Pdmn} From \cite{Vis98}*{pg. 638} we know that there is a surjective morphism:
\[ \rho^*:\Pic(\Gcal_{d,m}^n) \arr \Pic(\Pcal\dmn^n)\]
whose kernel is generated by the first Chern class of the invertible sheaf $\Lcal\dmn$. This relation may be computed using \propref{prop:Gdmn quot}: indeed, the pullback of $\Lcal\dmn$ along the $\GL_{n+1}$ torsor $U\dmn\setminus D\dmn \arr \Gcal_{d,m}^n$ is the equivariant invertible sheaf \[\pr|_{X\dmn*}(\det(\Ical_{X\dmn}|_{X\dmn})\otimes\pr_2^*\Ocal(mn)|_{X\dmn})\otimes\det(E_{n+1})\]
where $X\dmn\arr U\dmn\setminus D\dmn $ is tautological family of smooth global complete intersections of three quadrics and $E_{n+1}$ is the standard representation of $\GL_{n+1}$.

Recall that $\Ical_{X\dmn}|_{X\dmn}\simeq \pr_1^*\Tcal\otimes\pr_2^*\Ocal(-d)|_{X\dmn}$, where $\Tcal$ is the universal subsheaf defined over $\Gras$. Putting all together, after a straightforward computation, we deduce:
\[ \Pic(\Pcal\dmn^n)\simeq\Pic(\Gcal_{d,m}^n)/\langle c_1-s_1 \rangle \]
where $s_1$ is the first special Schubert cycle, i.e. the first Segre class of $\Tcal$. We are ready to prove the main result of this subsection:

\begin{thm}\label{thm:pic M3}
	Let $\lambda_1$ be the first Chern class of the Hodge bundle over $\Mcal_3$. Then the Picard group of $\Mcal_3$ is freely generated by $\lambda_1$, without any assumption on the characteristic of the base field.
\end{thm}
\begin{proof}
    Putting together \propref{prop:U3 iso P41}, \parref{par:Pic Pdmn} and \thmref{thm:pic Gdmn} we deduce that:
	\[ \Pic(\Ucal_3)=\ZZ\langle c_1 \rangle/\langle 9c_1 \rangle \]
	Consider the localization exact sequence:
	\[\ZZ\cdot[\Hcal_3]\arr \Pic(\Mcal_3)\arr \Pic(\Ucal_3) \arr 0 \]
	From this we see that the cycle class of $\Hcal_3$ in $\Pic(\Mcal_3)$ is equal to $9\c_1$ and $\Pic(\Mcal_3)$ is generated by $c_1$.
	
	Consider the category fibred in groupoids $\wt{\Mcal}_3$ whose objects are pairs $(\pi:C\arr S,\alpha)$, where $\pi:C\arr S$ is an object of $\Mcal_3$ and $\alpha$ is an isomorphism between $\pi_*\omega_\pi$ and $\Ocal_S^{\oplus 3}$. 
	
	We see that $\wt{\Mcal}_3$ is a $\GL_3$-torsor over $\Mcal_3$, thus it induces a morphism $p:\Mcal_3\arr \Bcal(\GL_3)$. By construction, if we denote $\Ecal=[E/\GL_3]$ the universal $\GL_3$-torsor over $\Bcal(\GL_3)$, we have that $p^*\Ecal=\wt{\Mcal}_3$.
	
	Therefore, the cycle $c_1$ is equal to the first Chern class of the locally free sheaf of rank three associated to the torsor $\wt{\Mcal}_3$, which is precisely the Hodge bundle. This shows that $c_1=\lambda_1$ and it concludes the proof.
\end{proof}
\begin{cor}\label{cor:H3}
    We have $[\Hcal_3]=9\lambda_1$.
\end{cor}
\subsection{Genus four case}
~\\
\mypar \label{par:U4'} Let $\Mcal_4$ be the moduli stack of smooth curves of genus four, and denote $\Ucal_4$ the open substack of non-hyperelliptic curves. Observe that the complement of $\Ucal_4$ has codimension two, thus $\Pic(\Mcal_4)\simeq\Pic(\Ucal_4)$. 
Thanks to \lmref{lm:base change}, we can define a fibred category $\Ucal'_4$ over the site of schemes whose objects are triples $(\pi, i, \beta)$, where:
\begin{itemize}
	\item $\pi:C\arr S$ is an object of $\Ucal_4$.
	\item $i:C\hookrightarrow\PP(\pi_*\omega_\pi)$ is a closed embedding.
	\item $\beta$ is an isomorphism between $i^*\Ocal(1)$ and $\omega_\pi$.
\end{itemize}
There is an obvious morphism $\Ucal_4'\arr \Ucal_4$. We also have a morphism in the opposite direction: indeed, given a smooth family of genus four, non-hyperelliptic curves $\pi:C\arr S$, by \lmref{lm:base change} we have a canonical, surjective morphism $\pi^*\pi_*\omega_\pi\arr \omega_\pi$, which in turn induces a canonical embedding $i:C\hookrightarrow\PP(\pi_*\omega_\pi)$ and an isomorphism $\beta:i^*\Ocal(1)\simeq\omega_\pi$. As everything is functorial, we get the claimed morphism $\Ucal_4 \arr \Ucal'_4$. It is almost immediate to check that the two morphisms are equivalences of stacks.

\mypar Observe that there is an invertible sheaf defined over the stack $\Fcal_{a,b}^n$, functorially defined as follows:
\[ \Lcal_{a,b}: (V\arr S,X\subset \PP(V)) \longmapsto \pr|_{X*}(\det(\Ical_X|_X)\otimes\Ocal_X(a+b))\otimes\det(V) \] 
where $\pr:\PP(V)\arr S$ is the canonical projection. The fact that the sheaf on the right is invertible easily follows from the cohomology and base change theorem. Let $\Qcal_{a,b}^n$ be the $\GG_m$-torsor associated to $\Lcal_{a,b}$. By definition, the objects of $\Qcal_{a,b}^n$ are triples $(V\arr S,X\subset\PP(V),\gamma)$, where $\gamma$ is an isomorphism between $\pr|_{X*}(\det(\Ical_X|_X)\otimes\Ocal_X(a+b))\otimes\det(V)$ and $\Ocal_S$.

\begin{prop}\label{prop:U4 iso P23}
	We have that $\Ucal_4\simeq\Qcal_{2,3}^3$.
\end{prop}
\begin{proof}
	From \parref{par:U4'}, we see that we can equivalently show that $\Ucal_4'\simeq\Qcal_{2,3}^3$.
	We construct a morphism $\Ucal'_4\arr\Qcal_{2,3}^3$ as follows: given an object $(\pi, i, \beta)$ of $\Ucal'_4$, we send it to the object $(\VV(\pi_*\omega_\pi)\arr S,i(C)\subset\PP(\pi_*\omega_\pi),\gamma)$, where $\VV(\pi_*\omega_\pi)$ denotes the total space of the vector bundle associated to the locally free sheaf $\pi_*\omega_\pi$. We are using here the well known fact that $i(C)$ is a family of smooth complete intersections of bidegree $(2,3)$. The element $\gamma$ is constructed as follows: the isomorphism $\beta$ can be seen as a global, everywhere non-vanishing, section of $i^*\Ocal(-1)\otimes\omega_\pi$. We have the following chain of easy identifications:
	\begin{align*}
	H^0(C,i^*\Ocal(-1)\otimes\omega_\pi)&=H^0(\PP(\pi_*\omega_\pi),i_*(i^*(\Ocal(-1)\otimes \Ical_C^{\vee} \otimes\omega_{\PP(\pi_*\omega_\pi)})))\\
	&= H^0(\PP(\pi_*\omega_\pi), \Ical_C^{\vee}\otimes \omega_{\PP(\pi_*\omega_\pi)}(-1))
	\end{align*}
	Using the canonical isomorphism $\omega_{\PP(\pi_*\omega_\pi)}\simeq \Ocal(-4)\otimes\det((\pi_*\omega_\pi)^\vee)$, we deduce that the isomorphism $\beta$ can actually be regarded as a trivialization $\gamma$ of the invertible sheaf $\pi_*(i^*(\Ical_C(5)))\otimes\det(\pi_*\omega_\pi)$. As everything is functorial, we have defined a morphism $\Ucal'_4\arr\Qcal_{2,3}^3$.
	
	To define the inverse morphism $\Qcal_{2,3}^3\arr \Ucal'_4$, observe that given an object $(V\arr S,i:X\subset \PP(V),\gamma)$ of $\Qcal_{2,3}^3$, we can obtain an isomorphism $\beta:i^*\Ocal(1)\simeq\omega_{X/S}$ from $\gamma$ by simply going backward in the chain of identifications above. Pushing forward $\beta$ to $S$, we obtain an isomorphism between $V$ and $\pr|_{X*}\omega_{X/S}$. Thus it makes sense to define a morphism $\Qcal_{2,3}^3\arr \Ucal_4'$ by sending a triple $(V\arr S,i:X\subset \PP(V),\gamma)$ to the triple $(\pr|_X:X\arr S,i:X\hookrightarrow\PP(V)\simeq\PP(\pr_*\omega_{X/S}),\beta)$. It is easy to see that the two morphisms that we have defined are equivalences of stacks.
\end{proof}

\mypar From \cite{Vis98}*{pg. 638} we know that the pullback morphism:
\[ \rho^*: \Pic(\Fcal_{a,b}^n)\arr \Pic(\Qcal_{a,b}^n) \]
is surjective, with kernel equal to the first Chern class of $\Lcal_{a,b}$. Recall from \propref{prop:Fcal quot} that we have an isomorphism between $\Fcal_{a,b}^n$ and $[U_{a,b}\setminus D_{a,b}/\GL_4]$. Call $\pi:C_{a,b}\arr U_{a,b}$ the universal family of complete intersections of bidegree $(a,b)$: by definition it is a closed subscheme of $U_{a,b}\times\PP^3$. If $\Ical_{C_{a,b}}$ denotes its sheaf of ideals, then we have:
\[ \Lcal_{a,b}\simeq [\pi_*(\det(\Ical_{C_{a,b}}|_{C_{a,b}})\otimes\pr_2^*\Ocal(a+b)|_{C_{a,b}})\otimes\det(E))/\GL_4 ] \]
Recall that: \[\Ical_{C_{a,b}}|_{C_{a,b}}=(\pr_1^*\Ocal_{\PP(W_a)}(-1)\otimes\pr_2^*\Ocal_{\PP^3}(-a))\oplus(\pr_1^*\Ocal_{\PP(V_{a,b})}(-1)\otimes\pr_2^*\Ocal_{\PP^3}(-b))|_{C_{a,b}}\]
Therefore $\det(\Ical_{C_{a,b}}|_{C_{a,b}})\otimes\pr_2^*\Ocal(a+b)|_{C_{a,b}}=\pr_1^*(\Ocal_{\PP(W_a)}(-1)\otimes\Ocal_{\PP(V_{a,b})}(-1))|_{C_{a,b}}$ and we deduce:
\[ \Pic(\Qcal_{a,b}^n)\simeq \Pic(\Fcal_{a,b}^n)/\langle c_1-u-v \rangle \]
where $u=c_1(\Ocal_{\PP(W_a)}(1))$ and $v=c_1(\Ocal_{\PP(V_{a,b})}(1))$.

\begin{thm}\label{thm:pic M4}
	Let $\lambda_1$ be the first Chern class of the Hodge bundle over $\Mcal_4$. Then the Picard group of $\Mcal_4$ is freely generated by $\lambda_1$, without any assumption on the characteristic of the base field.
\end{thm}
\begin{proof}
	From \parref{par:U4'} and \propref{prop:U4 iso P23} we know that:
	\[ \Pic(\Mcal_4)=\Pic(\Fcal_{2,3}^3)/\langle c_1-u-v \rangle \]
	Applying \thmref{thm:Pic(Fa,b)} when $a=2$, $b=3$ and $n=3$ we obtain:
	\[ \Pic(\Fcal_{2,3}^3)=\ZZ\langle c_1,u,v \rangle/\langle33u+34v-42c_1\rangle \]
	Therefore we get that $\Pic(\Mcal_4)$ is freely generated by $c_1$. By definition $c_1$ is the first Chern class of the vector bundle associated to the $\GL_4$-torsor $\wt{\Fcal}_{2,3}^3$ that we introduced in the proof of \propref{prop:Fcal quot}. If we pull back this torsor along the morphism $\Ucal_4\arr \Fcal_{2,3}^3$ that we have constructed in \parref{par:U4'} and \propref{prop:U4 iso P23}, we get exactly the $\GL_4$-torsor associated to the Hodge bundle. This implies that $c_1=\lambda_1$ and concludes the proof.
\end{proof}

\mypar Let $\Mcal_4^{\rm ev}$ denote the closed substack of $\Mcal_4$ that parametrizes those curves having an even theta characteristic. It is well known that $\Mcal_4^{\rm ev}$ has codimension one. Therefore, to compute the class $[\Mcal_4^{\rm ev}]$ in the Picard group of $\Mcal_4$, we can equivalently compute the class of its restriction to $\Pic(\Ucal_4)$.

Recall that a family of smooth curves of genus four having an even theta characteristic is canonically embedded as a complete intersection of a rank three quadric and a cubic. Let $\Delta_2$ denotes the divisor in $\PP(W_2)$ parametrizing rank three quadrics: then it follows that, in order to compute $[\Mcal_4^{rm ev}]$, we only have to determine the cycle class of $\Delta_2$ in $\Pic^{\GL_4}(\PP(W_2))$ and then use the relations $33u+34v-42c_1=0$ and $c_1-u-v=0$ to reduce the expression that we had found to a multiple of $c_1=\lambda_1$.

Applying \cite{FulVis}*{Pr. 4.3} we obtain that $[\Delta_2]=4u-c_1$. Putting everything together, we deduce the following result of Teixidor i Bigas (see \cite{Tei}*{Pr. 3.1}).

\begin{cor}\label{cor:M4ev}
	We have $[\Mcal_4^{\rm ev}]=34\lambda_1$.
\end{cor}

\subsection{Genus five case}
~\\
\mypar \label{par:U5'} Let $\Mcal_5$ be the moduli stack of smooth curves of genus five and let $\Tcal_5$ denotes the closed substack of trigonal curves. It is well known that this closed substack has codimension one. Let $\Ucal_5$ be the complement of $\Tcal_5$ in $\Mcal_5$, i.e. the moduli stack of smooth, non-trigonal curves of genus five.

Again by \lmref{lm:base change}, we can consider the fibred category $\Ucal'_5$ over the site of schemes whose objects are triples $(\pi, i, \beta)$, where:
\begin{itemize}
	\item $\pi:C\arr S$ is an object of $\Ucal_5$.
	\item $i:C\hookrightarrow\PP(\pi_*\omega_\pi)$ is a closed embedding.
	\item $\beta$ is an isomorphism between $i^*\Ocal(1)$ and $\omega_\pi$.
\end{itemize}
Using the same argument of \parref{par:U4'} we see that the two stacks $\Ucal_5$ and $\Ucal_5'$ are equivalent.

\mypar Recall from \parref{par:Pdmn} that there is a $\GG_m$-torsor $\Pcal\dmn^n$ over $\Gcal_{d,m}^n$ whose objects are triples $(V\arr S,X\subset\PP(V),\gamma)$, where the pair $(V\arr S,X\subset\PP(V))$ is an object of $\Gcal_{d,m}^n$ and $\gamma$ is a trivializing section of $\pr_*(\det(\Ical_X|_X)\otimes\Ocal_X(mn))\otimes\det(V)$.

\begin{prop}\label{prop:U5 iso P234}
	Set $n=4$. Then we have $\Ucal_5\simeq \Pcal_{2,3}^4$.
\end{prop}
\begin{proof}
	From \parref{par:U5'} we see that we can equivalently show that $\Ucal_5'$ is isomorphic to $\Pcal_{2,3}^4$.
	First we construct a morphism $\Ucal_5'\arr \Pcal_{2,3}^4$. Recall that the canonical model of a family $\pi:C\arr S$ of smooth, non-trigonal curves of genus five is a smooth complete intersection of three quadrics, thus given an object $(\pi,i,\beta)$ of $\Ucal_5'$ the pair $(\VV(\pi_*\omega_\pi)\arr S,i(C)\subset\PP(\pi_*\omega_\pi))$ is an object of $\Gcal_{2,3}^4$.
	
	Moreover, using the same argument of the proof of \propref{prop:U4 iso P23} we see that the isomorphism $\beta$ induces a trivializing section of $\pi_*(\det(\Ical_{C}|_C)\otimes\Ocal_C(6))\otimes\det(\pi_*\omega_\pi)$. As everything is functorially well behaved, this defines a morphism $\Ucal_5'\arr \Pcal_{2,3}^4$.
	
	To construct the inverse morphism, consider an object $(V\arr S, X\subset \PP(V),\gamma)$ of $\Pcal_{2,3}^4$: as in the proof of \propref{prop:U4 iso P23}, the trivializing section $\gamma$ induces an isomorphism $\beta:i^*\Ocal(1)\simeq \omega_{X/S}$, that pushed forward to $S$ allows us to identify $V$ with $\pr|_{X*}\omega_{X/S}$, using the canonical isomorphism $\pr_*\Ocal(1)\simeq V$. We use this identification to define the closed embedding $i:X\subset\PP(V)\simeq\PP(\pr|_{X*}\omega_{X/S})$. Therefore, we can construct the morphism $\Pcal_{2,3}^4\arr \Ucal_5'$ by sending a triple $(V\arr S, X\subset\PP(V),\gamma)$ to the object $(\pr|_X,i,\beta)$.
	
	It is immediate to check that the two morphisms that we have defined are equivalences of stacks.
\end{proof}
We are ready to prove tha main theorem of this subsection:

\begin{thm}\label{thm:pic M5}
	Let $\lambda_1$ be the first Chern class of the Hodge bundle over $\Mcal_5$. Then the Picard group of $\Mcal_5$ is freely generated by $\lambda_1$, without any assumption on the characteristic of the base field.
\end{thm}
\begin{proof}
	Putting together \propref{prop:U5 iso P234}, \parref{par:Pic Pdmn} and \thmref{thm:pic Gdmn} we deduce that:
	\[ \Pic(\Ucal_5)=\ZZ\langle c_1 \rangle/\langle 8c_1 \rangle \]
	From the exact sequence
	\[\ZZ\cdot[\Tcal_5]\arr \Pic(\Mcal_5) \arr \Pic(\Ucal_5) \arr 0 \]
	we easily conclude that $\Pic(\Mcal_5)$ is freely generated by $c_1$. The cycle $c_1$ comes from the Picard group of $\Gcal_{2,3}^4$: in \propref{prop:Gdmn quot} we showed in particular that $\Gcal_{2,3}^4$ has a $\GL_5$-torsor over it, which is the scheme $U_{2,3}\setminus D_{2,3}$, that can be described as the stack in sets whose objects are triples $(V\arr S,X\subset\PP(V),\alpha)$ where $(V\arr S,X\subset\PP(V))$ is an object of $\Gcal_{2,3}^4$ and $\alpha$ is an isomorphism between the locally free sheaf associated to $V$ and $\Ocal_S^{\oplus 5}$.
	
	The cycle $c_1$ is the first Chern class of the locally free sheaf associated to this $\GL_5$-torsor. This locally free sheaf can be described as the functor:
	\[ (V\arr S,X\subset\PP(V)) \longmapsto (V\arr S) \]
	It is immediate to check that if we pull back this sheaf along the morphism $\Ucal_5\arr \Gcal_{2,3}^4$ that we constructed in the proof of \propref{prop:U5 iso P234} we recover the Hodge bundle restricted to $\Ucal_5$, thus $c_1=\lambda_1$. This concludes the proof of the theorem.
\end{proof}
In particular, from the proof above we can retrieve a particular case of \cite{HM}*{pg. 24}:

\begin{cor}\label{cor:T5}
    We have $[\Tcal_5]=8\lambda_1$.
\end{cor}


\begin{bibdiv}
	\begin{biblist}
		\bib{AC}{article}{
			author={Arbarello, Enrico},
			author={Cornalba, Maurizio},
			title={The Picard groups of the moduli spaces of curves},
			journal={Topology},
			volume={26},
			date={1987},
			number={2},
			pages={153--171},
			issn={0040-9383},
			review={\MR{895568}},
			doi={10.1016/0040-9383(87)90056-5},
		}
        \bib{Ben}{article}{
           author={Benoist, Olivier},
           title={Degr\'{e}s d'homog\'{e}n\'{e}it\'{e} de l'ensemble des intersections compl\`etes
           singuli\`eres},
           language={French, with English and French summaries},
           journal={Ann. Inst. Fourier (Grenoble)},
           volume={62},
           date={2012},
           number={3},
           pages={1189--1214}
        }
		\bib{Dil18}{article}{
		    author={Di Lorenzo, Andrea},
		    title={The Chow ring of the stack of hyperelliptic curves of odd genus},
		    journal={arXiv:1802.04519 [math.AG]},
		    date={2018}
		}
		\bib{EF08}{article}{
			author={Edidin, Dan},
			author={Fulghesu, Damiano},
			title={The integral Chow ring of the stack of at most 1-nodal rational
				curves},
			journal={Comm. Algebra},
			volume={36},
			date={2008},
			number={2},
			pages={581--594},
			issn={0092-7872},
			review={\MR{2388024}},
			doi={10.1080/00927870701719045},
		}
		\bib{EF09}{article}{
			author={Edidin, Dan},
			author={Fulghesu, Damiano},
			title={The integral Chow ring of the stack of hyperelliptic curves of
				even genus},
			journal={Math. Res. Lett.},
			volume={16},
			date={2009},
			number={1},
			pages={27--40},
			issn={1073-2780},
			review={\MR{2480558}},
			doi={10.4310/MRL.2009.v16.n1.a4},
		}
		\bib{EG}{article}{
			author={Edidin, Dan},
			author={Graham, William},
			title={Equivariant intersection theory},
			journal={Invent. Math.},
			volume={131},
			date={1998},
			number={3},
			pages={595--634},
			issn={0020-9910},
			review={\MR{1614555}},
			doi={10.1007/s002220050214},
		}

		\bib{m3bar}{article}{
			author={Faber, Carel},
			title={Chow rings of moduli spaces of curves. I. The Chow ring of
				$\overline{\scr M}_3$},
			journal={Ann. of Math. (2)},
			volume={132},
			date={1990},
			number={2},
			pages={331--419},
			issn={0003-486X},
			review={\MR{1070600}},
			doi={10.2307/1971525},
		}
	
		\bib{m4}{article}{
			author={Faber, Carel},
			title={Chow rings of moduli spaces of curves. II. Some results on the
			Chow ring of $\overline{\scr M}_4$},
			journal={Ann. of Math. (2)},
			volume={132},
			date={1990},
			number={3},
			pages={421--449},
			issn={0003-486X},
			review={\MR{1078265}},
			doi={10.2307/1971526},
		}
	
			
        \bib{Ful}{book}{
           author={Fulton, William},
           title={Intersection theory},
           series={Ergebnisse der Mathematik und ihrer Grenzgebiete. 3. Folge. A
           Series of Modern Surveys in Mathematics [Results in Mathematics and
           Related Areas. 3rd Series. A Series of Modern Surveys in Mathematics]},
           volume={2},
           edition={2},
           publisher={Springer-Verlag, Berlin},
           date={1998},
           pages={xiv+470},
           isbn={3-540-62046-X},
           isbn={0-387-98549-2},
           review={\MR{1644323}},
           doi={10.1007/978-1-4612-1700-8},
        }

		\bib{FulVis}{article}{
			author={Fulghesu, Damiano},
			author={Vistoli, Angelo},
			title={The Chow ring of the stack of smooth plane cubics},
			journal={Michigan Math. J.},
			volume={67},
			date={2018},
			number={1},
			pages={3--29},
			issn={0026-2285},
			review={\MR{3770851}},
			doi={10.1307/mmj/1516330968},
		}
		
    	\bib{Har}{article}{
    		author={Harer, John},
    		title={The second homology group of the mapping class group of an
    			orientable surface},
    		journal={Invent. Math.},
    		volume={72},
    		date={1983},
    		number={2},
    		pages={221--239},
    		issn={0020-9910},
    		review={\MR{700769}},
    		doi={10.1007/BF01389321},
    	}
    	
    	\bib{HM}{article}{
           author={Harris, Joe},
           author={Mumford, David},
           title={On the Kodaira dimension of the moduli space of curves},
           note={With an appendix by William Fulton},
           journal={Invent. Math.},
           volume={67},
           date={1982},
           number={1},
           pages={23--88},
        }
    
    	\bib{Iza}{article}{
    		author={Izadi, E.},
    		title={The Chow ring of the moduli space of curves of genus $5$},
    		conference={
    			title={The moduli space of curves},
    			address={Texel Island},
    			date={1994},
    		},
    		book={
    			series={Progr. Math.},
    			volume={129},
    			publisher={Birkh\"{a}user Boston, Boston, MA},
    		},
    		date={1995},
    		pages={267--304},
    		review={\MR{1363060}},
    	}
    	
        \bib{Kre}{article}{
           author={Kresch, Andrew},
           title={Cycle groups for Artin stacks},
           journal={Invent. Math.},
           volume={138},
           date={1999},
           number={3},
           pages={495--536},
        }

    	\bib{Mor}{article}{
    		author={Moriwaki, Atsushi},
    		title={The $\mathbb{Q}$-Picard group of the moduli space of curves in
    			positive characteristic},
    		journal={Internat. J. Math.},
    		volume={12},
    		date={2001},
    		number={5},
    		pages={519--534},
    		issn={0129-167X},
    		review={\MR{1843864}},
    		doi={10.1142/S0129167X01000964},
    	}
    	
		\bib{Mum63}{article}{
			author={Mumford, David},
			title={Picard groups of moduli problems},
			conference={
				title={Arithmetical Algebraic Geometry},
				address={Proc. Conf. Purdue Univ.},
				date={1963},
			},
			book={
				publisher={Harper \& Row, New York},
			},
			date={1965},
			pages={33--81},
			review={\MR{0201443}},
		}
		
		\bib{Ols}{book}{
           author={Olsson, Martin},
           title={Algebraic spaces and stacks},
           series={American Mathematical Society Colloquium Publications},
           volume={62},
           publisher={American Mathematical Society, Providence, RI},
           date={2016}
        }
    	
    	\bib{PV}{article}{
    		author={Penev, Nikola},
    		author={Vakil, Ravi},
    		title={The Chow ring of the moduli space of curves of genus six},
    		journal={Algebr. Geom.},
    		volume={2},
    		date={2015},
    		number={1},
    		pages={123--136},
    		review={\MR{3322200}},
    		doi={10.14231/AG-2015-006},
    	}
		
        \bib{SGA}{book}{
           label={SGA VII}
           title={Groupes de monodromie en g\'{e}om\'{e}trie alg\'{e}brique. II},
           series={Lecture Notes in Mathematics, Vol. 340},
           note={S\'{e}minaire de G\'{e}om\'{e}trie Alg\'{e}brique du Bois-Marie 1967--1969 (SGA 7
           II);
           Dirig\'{e} par P. Deligne et N. Katz},
           publisher={Springer-Verlag, Berlin-New York},
           date={1973}
        }
        \bib{Tei}{article}{
           author={Teixidor i Bigas, Montserrat},
           title={The divisor of curves with a vanishing theta-null},
           journal={Compositio Math.},
           volume={66},
           date={1988},
           number={1},
           pages={15--22}
        }
    	\bib{Tot}{article}{
			author={Totaro, Burt},
			title={The Chow ring of a classifying space},
			conference={
				title={Algebraic $K$-theory},
				address={Seattle, WA},
				date={1997},
			},
			book={
				series={Proc. Sympos. Pure Math.},
				volume={67},
				publisher={Amer. Math. Soc., Providence, RI},
			},
			date={1999},
			pages={249--281},
			review={\MR{1743244}},
			doi={10.1090/pspum/067/1743244},
		}
	    \bib{Vis89}{article}{
	    	author={Vistoli, Angelo},
	    	title={Intersection theory on algebraic stacks and on their moduli
	    		spaces},
	    	journal={Invent. Math.},
	    	volume={97},
	    	date={1989},
	    	number={3},
	    	pages={613--670},
	    	issn={0020-9910},
	    	review={\MR{1005008}},
	    	doi={10.1007/BF01388892},
	    }
    	
    	
    	\bib{Vis98}{article}{
    		author={Vistoli, Angelo},
    		title={The Chow ring of $\scr M_2$. Appendix to ``Equivariant
    			intersection theory'' [Invent.\ Math.\ {\bf 131} (1998), no.\ 3,
    			595--634; MR1614555 (99j:14003a)] by D. Edidin and W. Graham},
    		journal={Invent. Math.},
    		volume={131},
    		date={1998},
    		number={3},
    		pages={635--644},
    		issn={0020-9910},
    		review={\MR{1614559}},
    		doi={10.1007/s002220050215},
    	}
    	
	\end{biblist}
\end{bibdiv}
\end{document}